\documentclass{amsart}
\usepackage{amsmath}
\usepackage{amssymb}
\usepackage{amsthm}
\usepackage{amsfonts}
\usepackage{url}
\usepackage{hyperref}
\usepackage{appendix}
\usepackage{verbatim}
\usepackage{graphicx,color}
\usepackage{lineno}

\usepackage{thmtools,  thm-restate} 

\title[Grothendieck pairs and uncountable ambiguity]
{Relatively hyperbolic groups,  Grothendieck pairs,  and uncountable profinite ambiguity among fibre products}

\author{M.R. ~Bridson}
\author{A.W. ~Reid} 
\address{\newline Mathematical Institute, 
\newline Andrew Wiles Building,
\newline University of Oxford,
\newline Oxford OX2 6GG, UK}
\email{bridson@maths.ox.ac.uk}
\address{\newline Department of Mathematics,
\newline Rice University, 
\newline Houston, TX 77005, USA.
\newline School of Mathematics,
\newline Korea Institute for Advanced Study (KIAS), Seoul, 02455, Korea.}
\email{alan.reid@rice.edu}
 
\thanks{The second author was supported by  NSF grant DMS-1812397.  For the purpose of open access, the authors have applied a CC BY public copyright licence to any author accepted manuscript.}

\def\-{\overline}

\def\wh{\widehat}

\def\G{\Gamma}

\def\ssm{\smallsetminus}

\def\La{\Lambda}

 \def\H{\mathbb{H}}
 \def\S{\mathbb{S}}
 \def\Z{\mathbb{Z}}
 \def\R{\mathbb{R}}
 \def\Q{\mathbb{Q}} 
  
 \def\C{\mathbb{C}}

 \def\dirlim{\lim}

\DeclareMathOperator{\SL}{SL} \DeclareMathOperator{\PSL}{PSL}

 \DeclareMathOperator{\PSU}{PSU}

\def\tr{\mbox{\rm{tr}}\, }

\def\G{\Gamma}

\def\<{\langle}
\def\>{\rangle}
\def\ilim{\varprojlim}

\def\GW{\Gamma_{W}}

\def\onto{\twoheadrightarrow} 

\newtheorem{theorem}{Theorem}[section]
\newtheorem{lemma}[theorem]{Lemma}
\newtheorem{corollary}[theorem]{Corollary}

\newtheorem{proposition}[theorem]{Proposition}

\theoremstyle{definition} 
\newtheorem{definition}[theorem]{Definition}
\newtheorem{remark}[theorem]{Remark}

\newtheorem{question}[theorem]{Question}

\begin{document}

\date{17 July 2025}

\begin{abstract} These notes expand upon our lectures on {\em profinite rigidity}
at the international colloquium on randomness, geometry and dynamics, 
organised by TIFR Mumbai at IISER Pune in January 2024.  We are interested in the extent to which groups that arise in 
hyperbolic geometry and 3-manifold topology
are determined by their finite quotients.  The main theme of  these notes is the 
radical extent to which rigidity is lost when one passes from consideration of groups 
with hyperbolic features to consideration of their direct products.
We describe a general method for producing infinite sequences of 
{\em{Grothendieck pairs,}} i.e.~embeddings  $P_i\hookrightarrow G\times G$ inducing isomorphisms of 
profinite completions,  with $G$ fixed  and $P_i$ finitely generated.
In order to apply this method, one needs $G$ to map onto a subgroup of finite index in the commutator subgroup of a group
$\Gamma$ with $H_2(\Gamma,\Z)=0$, and $\Gamma$ should be 
relatively hyperbolic.  By exploiting the flexibility of the construction,
we explain how, under the same hypotheses on $G$, one
can construct {\em uncountable families} of pairwise non-isomorphic subgroups $P_\lambda$ such that
$P_\lambda\hookrightarrow G\times G$ induces an isomorphism of profinite completions.
Examples of groups $G$ satisfying these conditions include the fundamental group of the Weeks manifold
and the fundamental group of the 4-fold branch cover of the figure-8 knot complement. Both of these examples
are profinitely rigid in the absolute sense and in each case Grothendieck pairs account entirely
for the loss of profinite rigidity for $G\times G$: if  $H$ is a  finitely generated group  whose
profinite completion  is isomorphic to that of $G\times G$, then there is an embedding $H\hookrightarrow G\times G$
that is a Grothendieck pair.
\end{abstract}

 

\date{16 July, Cambridge}


\maketitle

%
%
%
%

\def\L{\Lambda}
\def\D{\Delta}
\def\hr{\hookrightarrow}

\section{Introduction}  

A finitely generated, residually finite group $G$ is {\em profinitely rigid} if any finitely generated, residually finite group $H$
with the same set of finite quotients as $G$ must be isomorphic to $G$; in more technical language, 
$\widehat{H}\cong\widehat{G}\implies H\cong G$, where $\widehat{G}$ is the profinite completion on $G$,
which is defined to be the inverse limit $\ilim  G/N$, where $N$ runs over the set of all finite-index normal
subgroups of $G$.
It is not difficult to show that $\Z^n$ is profinitely rigid, but it is much more difficult to construct examples that 
are full-sized in the sense that they
contain a non-abelian free group: we constructed the first such examples in our work with McReynolds and Spitler \cite{BMRS1, BMRS2}.  
Deviation from profinite rigidity is measured by the  {\em profinite genus} of $G$, which is the set of isomorphism 
classes of finitely generated, residually finite groups $H$ with $\widehat{H}\cong\widehat{G}$.  The   {\em strong profinite genus} of $G$  consists   of isomorphism classes of finitely generated 
groups $H$ for which there is an embedding $u:H\hookrightarrow G$ that induces
an isomorphism $\widehat{u}:\widehat{H}\hookrightarrow \widehat{G}$; if $H\neq G$
then $H\hr G$ is called {\em Grothendieck pair}.  
In \cite{BRS} we exhibited infinite families of Seifert fibered 3-manifolds $M$
with geometry ${\widetilde{{\rm{SL}}}}_2$ such that $\pi_1M$ is profinitely rigid but the
strong profinite genus of $\pi_1M\times \pi_1M$ is infinite.   
In the course of these notes, we will
touch on most of the ideas behind the main construction in \cite{BRS} and explain how they can be extended.

\def\GW{\Gamma_W}

The first examples  of full-sized groups that are profinitely rigid \cite{BMRS1, BMRS2} were
the fundamental groups of certain 2-dimensional and 3-dimensional hyperbolic orbifolds, 
i.e. lattices in ${\rm{PSL}}(2,\R)$ and ${\rm{PSL}}(2,\C)$.     
We  expect  all lattices in ${\rm{PSL}}(2,\R)$ and ${\rm{PSL}}(2,\R)$
to be profinitely rigid,  but our focus here is not on the programme to prove that.
Rather,  the main point that 
we want to emphasize here is how emphatically profinite rigidity fails when one passes from
consideration of hyperbolic groups such as these to consideration of direct products of hyperbolic groups.
The following special case of what we prove illustrates this point nicely.   We frame this theorem in terms of two of
the early examples of 3-manifold groups that were proved to be profinitely rigid:  
from \cite{BMRS1},  the fundamental group of the Weeks manifold,  which is the unique closed hyperbolic 3-manifold of minimal volume; and from \cite{prasad}, the  fundamental group of the hyperbolic manifold obtained by
taking the $4$-fold cyclic  
branched cover of $\S^3$ over the figure-8 knot.

\begin{restatable}{letterthm}{fig8}\label{t:fig8}
Let $\G$ be the fundamental group of either of the Weeks manifold or the 4-fold cyclic  
branched cover of
$\S^3$ over the figure-8 knot.  Then, 
\begin{enumerate}
\item $\G$ is profinitely rigid; 
\item for every finitely generated, residually finite group $\Lambda$, if $\wh{\Lambda}\cong\wh{\G}\times\wh{\G}$ then
there exists an embedding $\Lambda\hr \G\times\G$ inducing this isomorphism of profinite completions;
\item there exist uncountably many pairwise non-isomorphic subgroups $u_i: P_i\hr \G\times\G$
such that $\wh{u}_i: \wh{P}_i\to \wh{\G}\times\wh{\G}$ is an isomorphism, and
\item infinitely many of these $P_i$ are finitely generated.
\end{enumerate} 
\end{restatable} 
 
The proof that the groups $\G$ in Theorem \ref{t:fig8} 
are profinitely rigid is based on the method of {\em Galois rigidity} which we shall recall in
Section \ref{s:galois}.  This property imposes control on the ${\rm{PSL}}(2,\C)$ representations of any
finitely generated group with the same profinite completion as $\G$.
While our main theme is that profinite rigidity is
lost when one passes from $\G$ to $\G\times\G$,  it is nevertheless true that Galois
rigidity is inherited by $\G\times\G$, and in favourable circumstances 
this can be used to create an embedding into $\G\times\G$
of any abstract,   finitely generated, residually finite group with the same profinite completion as $\G$;
that is the content of item (2) in  Theorem \ref{t:fig8}.
Items (3) and (4),  in contrast,  are not specific  to delicate features of the 3-manifold groups under consideration:
they  follow from a much more general result concerning Grothendieck
pairs,  Theorem \ref{t:lots-of-P},
which is the   main original feature of these notes. 

The   technique  that we shall describe for constructing the
abundance of Grothendieck pairs  in Theorem \ref{t:lots-of-P} extends a well-established
 train of ideas which we will now explain.
 
Grothendieck \cite{groth} asked if there exist Grothendieck pairs of finitely presented groups. This problem was eventually solved by
Bridson and Grunewald \cite{BG}. Their proof builds in an essential way on an earlier argument of Platonov and Tavgen \cite{PT}
who constructed the first Grothendieck pair of finitely generated groups with an argument whose essence can be abstracted as follows.
\begin{proposition}\label{p:PT}
Let $f:G\to Q$ be an epimorphism of groups, with $G$ finitely generated and $Q$ finitely presented, and
consider the fibre product 
$P=\{(g,h) \mid f(g)=f(h)\} < G\times G$. Then,
\begin{enumerate}
\item $P$ is finitely generated;
\item if $\widehat{Q}=1$ and $H_2(Q,\Z)=0$, then $P\hookrightarrow G\times G$ induces an isomorphism $\widehat{P}\overset{\cong}\rightarrow\widehat{G\times G}$.
\end{enumerate}
\end{proposition} 
Platonov and Tavgen applied this criterion with
 $G$ a free group of rank $4$ 
 and $Q$ Higman's famous 4-generator, 4-relator group with $\widehat{Q}=1$ and $H_2(Q,\Z)=0$.
In the course of constructing counterexamples to the Platonov Conjecture, Bass and Lubotzky \cite{BL} described infinitely 
many Grothendieck pairs of the form $(\G\times\G, P)$ by applying the full force
of Proposition \ref{p:PT} (which they identified as the essence of \cite{PT}).
In order to achieve this, they needed  a technique for mapping hyperbolic groups onto finitely presented groups $Q$ with
$\wh{Q}=1$ and  $H_2(Q,\Z)=0$. This is achieved in two stages. First, a bespoke
small-cancellation argument due to Olshanskii \cite{Ol} (for which
Rips has an alternative, unpublished proof) shows that every non-elementary hyperbolic group $G$ maps onto a finitely presented group
with $\wh{Q}=1$, but one does not have control over $H_2(Q,\Z)$. If $G$ is free, 
or more generally $H^2(G,\Z)=0$, this is not a serious problem -- one can replace 
$Q$ by its universal central extension $\widetilde{Q}$ and lift $G\to Q$ to a surjection $G\to\widetilde{Q}$ -- but for an arbitrary hyperbolic
group $G$ one cannot do this. Indeed, it remains unclear whether an arbitrary
non-elementary hyperbolic group (virtually) maps onto any infinite, finitely presented
group $\widetilde{Q}$ with no finite quotients and $H_2(\widetilde{Q},\Z)=0$. (Bass and Lubotzky get around this by using specific homological
properties of the groups that were their concern.) Such struggles with $H_2(-,\Z)$ will recur as a theme throughout these
notes.

Advances in the understanding of hyperbolic groups and significant advances in the theory of relatively hyperbolic groups, mean that
today one can replace Olshanskii's carefully crafted argument with a more flexible  and conceptually-easier construction. 
In \cite{BRS} we used   this
flexibiity to prove a version of 
the following theorem,  whose proof relies heavily on ideas from \cite{amo} and \cite{abj}.  The new embellishment that we
want to exploit in these notes  lies with the technical condition on the torsion in the quotients $Q_i$.

\begin{restatable}{letterthm}{lotsQ}\label{t:lots-of-Q}  
Let $\G$ be a non-elementary relatively hyperbolic group, let $\pi_0$ be the set of primes $p$
such that $\G$ has an element of order $p$,  and let $p_1,p_2, p_3,\dots$ be an arbitrary sequence of primes.
Then,  there is an infinite sequence of  groups 
$Q_i$ and  epimorphisms $\G\to Q_1\to Q_2\to\dots $ such that
\begin{enumerate}
\item  $\widehat{Q}_i=1$ for $i=1,2,\dots$;
\item if $\G$ is finitely presented, then so is each $Q_i$; 
\item each $Q_i$ has trivial centre;
\item $Q_i$ has an element of prime order $p$  only if $p\in \pi_0\cup \{p_1,\dots, p_i\}$;
\item $Q_i$ is generated by elements of order $p$ for each prime $p\in \{p_1,\dots, p_i\}$;
\item $Q_i\not\cong Q_j$ if $i\neq j$.
\end{enumerate}
\end{restatable}

The point of arranging the quotients $Q_i$ into a directed system rather than just a set is that we
can form the direct limit $Q_\infty = \dirlim Q_i$ and consider the fibre product associated to $\G\onto Q_\infty$
as well as the fibre products associated to $\G\onto Q_i$.   Our goal is to obtain Grothendieck pairs  using these fibre products, as in  Proposition \ref{p:PT}.  The point of controlling the torsion with a sequence of
primes is that we need an invariant that ensures the groups $Q_\infty$ associated to different sequences are not isomorphic,
giving us uncountably many examples -- see the proof of Theorem \ref{t:lots-of-P}.

Before we can appeal to Proposition \ref{p:PT},  we have to force $H_2(Q_i,\Z)=0$.
It seems unlikely that one can do this in general,  but  using a homological argument 
presented in Lemma \ref{l:schur},
one can do this for subgroups of finite index  in the commutator subgroup  of a
 relatively hyperbolic group with trivial second homology.  

\begin{restatable}{letterthm}{lotsP}\label{t:lots-of-P}
Let $\G$ be a finitely presented, 
non-elementary, relatively hyperbolic group  
and suppose that for infinitely many primes $p$ there is not  
an element of order $p$ in $\G$.
Let $G$ be either a hyperbolic group  in which centralizers of non-trivial elements
are virtually cyclic,  or a central extension of such a group. 

If $H_2(\G,\Z)=0$ and  $G$  maps onto a subgroup of finite index in $[\G, \G]$, then
\begin{enumerate}
\item there exist uncountably many
non-isomorphic  groups $P_\lambda$ with embeddings $P_\lambda\hookrightarrow G\times G$
that induce an isomorphism of profinite completions, and 
\item infinitely many of these $P_\lambda$ are finitely generated.
\end{enumerate}
\end{restatable} 

The case where $\G=G$ in Theorem \ref{t:lots-of-P} restricts us to perfect groups, but there are also 
many other groups to which the main idea behind Theorem \ref{t:lots-of-P} applies: the key property
that we require of $G$ is that after arranging for it to map onto a group $Q$ with no finite quotients,  we need to be able to
lift the map to get a surjection $G\onto \widetilde{Q}$.  Finitely generated free groups $F$ have this property
and consequently we obtain uncountably many non-isomorphic subgroups $P_\lambda< F\times F$ 
for which the inclusion induces an isomorphism of profinite completions.  In this case, it was already 
known that there are infinitely many Grothendieck pairs $P\hr F\times F$ with $P$ finitely generated \cite{BG}.
In contrast, if $F$ is a finitely generated free, surface of limit group, then there are no 
Grothendieck pairs $P\hr F\times\cdots\times F$ with $P$ {\em finitely presented} \cite{BW}.  

We mention this
in the context of Theorem \ref{t:lots-of-P}, because we want to emphasize that the families of  finitely generated 
subgroups $P_\lambda\hr G\times G$ constructed  will not be finitely presented in
general.  There do exist hyperbolic groups $H$ in which one can use similar ideas to construct infinite sequences of 
Grothendieck pairs $P_n\hr H\times H$ with $P_n$ finitely presented \cite{mrb:jems}, but 
those $H$ are specially crafted to this end,  and we do not know how
to do this when $H$ is a profinitely rigid group.

Whilst discussing the importance of finiteness properties, we should also direct the reader's attention to the
fact that it has been known for some time that there exist uncountable families of finitely generated, residually finite groups
that have the same profinite completion \cite{pyber, nekr}, although of course this phenomenon cannot occur
among the subgroups of a fixed finitely generated (or countable) group.
 
 \medskip

 \noindent{\bf Acknowledgements:}~{\em The first author thanks Stanford University for hosting him for a 
 productive sabbatical and the second author   thanks the Korean Institute for Advanced Study (KIAS) for its support and hospitality.
Both authors thank TIFR Mumbai and IISER Pune for the invitations to speak in January $2024$.
We are also grateful to ICMAT Madrid and the Isaac Newton Institute in Cambridge for hosting us at various stages of the writing of this paper.}

\section{Preliminaries}

\subsection{Profinite completions}

The finite quotients of a group $\G$ form a directed system,  since $\G/M\onto\G/N$ if $N>M$. 
The  {\em profinite completion} of  $\G$ is the inverse limit of this system. $\G$ inherits
a topology  from the discrete topology on the finite groups $\G/N$ and in this topology it is a
compact, totally-disconnected group.  The natural map
$\G\to \wh{\G}$ is injective if and only if $\G$ is residually finite.  A morphism of discrete groups
$u:\G_1\to \G_2$ induces a continuous morphism $\wh{u}:\wh{\G}_1\to \wh{\G}_2$. We are interested
in knowing when the inclusion of a subgroup $P\hr \G$ induces an isomorphism $ \wh{P}\cong \wh{\G}$.
When $\G$ is finitely generated, this amounts to proving two things: (1) for surjectivity,  that for every
finite group $G$ and every epimorphism $\pi:\G\to G$, the restriction of $\pi$ to $P$ is surjective; and (2)
for injectivity,   every epimorphism $P\onto G$ extends to an epimorphism  from $\G$.

A basic theorem in the field that we shall use without further comment is the following:  finitely generated
groups $\G_1$ and $\G_2$ have the same set of finite quotients if and only if $\wh{\G}_1\cong\wh{\G}_2$ (where the isomorphism can be taken to be an abstract isomorphism due to work of Nikolov and Segal \cite{NS}).

\subsection{Fibre products}

The fibre product $P<\G\times\G$ associated to an epimorphism $\pi: \G\onto Q$ is the subgroup
$$
P = \{ (g,h) \mid \pi(g)=\pi(h)\}.
$$
$P$ is a normal subgroup if and only if  $Q$ is abelian.  If $\G$ is finitely generated and $Q$ is finitely presented, 
then $P$ will be finitely generated (the 0-1-2 Lemma).  The question of when $P$ is finitely presented
is more delicate: it suffices that $\G$ is finitely presented,  $K=\ker\pi$ is finitely generated and $Q$
has a classifying space   with finite 3-skeleton (the 1-2-3 Theorem \cite{BBMS}).  The fibre products
that we will be considering do not satisfy the conditions of the 1-2-3 Theorem and there is no reason to expect them to be finitely presented.

Notable subgroups in $P$ include the diagonal   $\Delta_\G=\{ (g,g) \mid g\in \G\}\cong \G$
and the intersections of $P$ with $\G\times 1$ and $1\times \G$, which are $K_1=K\times 1$ and $K_2=1\times K$.
The semidirect product decomposition  $P = K_1\rtimes\Delta_\G$ (symmetrically,  $ K_2\rtimes \Delta_\G$) is useful,
as is the observation that the action of $(g,g)\in \Delta_\G$ by conjugation on $K\times 1$ is the same as the
action of $(g,1)$, so  the following groups of coinvariants (defined in the next subsection) coincide
\begin{equation}
H_0(\Delta_\G, \,  H_1(K_1,\Z)) \cong H_0(\G, \,  H_1(K,\Z)).
\end{equation}
Another useful observation is that $P/K_1K_2\cong Q$,
because the image of $P$ under $(\pi, \pi): \G\times \G\to Q\times Q$ is the diagonal subgroup.

\subsection{On the homology of groups}

We shall assume that the reader is familiar with basic facts concerning the homology of groups, for which we refer to \cite{brown}.  In particular, the reader will need to recall that $H_1(G,\Z)$ is the abelianisation of $G$, and that
if $N$ is a normal subgroup of a group $G$, then the group of co-invariants $H_0(G, \, H_1(N,\Z))$ is the quotient
of $H_1(N,\Z)$ obtained by factoring out the conjugation action of $G$, i.e. ~killing the subgroup 
$\<[n] - [g^{-1}ng] \mid n\in N,\, g\in G\>$. We shall also need the 5-term exact sequence associated to a
short exact sequence $1\to N\to G\to Q\to 1$, which is 
$$
H_2(Q,\Z)\to H_0(G, \, H_1(N,\Z)) \to H_1(G,\Z)\to H_1(Q,\Z)\to 0.
$$

The following lemma will also be useful.

\begin{lemma}\label{l:balanced} Let $G$ be a group that has a finite presentation with equal numbers of generators and relations. If $H_1(G,\Z)$ is finite then $H_2(G,\Z)=0$.
\end{lemma} 

\begin{proof} The standard 2-complex $K$ of a presentation has  one vertex,  one 1-cell for each generator $a$,  oriented
and labelled $a$,  and one 
2-cell for each relator $r$ of the presentation: the attaching map of the 2-cell traces out the loop in the 1-skeleton labelled $r$.  A  classifying space  for the group presented $G$ can be obtained by adding cells of dimension greater than $2$, and hence
$H_2(G,\Z)$ is a quotient of $H_2(K,\Z)$, so in our setting we will be done if we can argue that $H_2(K,\Z)=0$. To see that 
$H_2(K,\Z)=0$ we consider the cellular chain complex of $K$,
$$
 C_2\to C_1 \to  C_0.
$$
By definition,  $C_1 \to  C_0$ is the zero map,  so $H_1(K,\Z)= H_1(G,\Z)$ will only be finite if
the image of $C_2\to C_1$ is of finite index.   By hypothesis, the free abelian groups $C_2$ and $C_1$ have the same rank,
so if the image of $d_2: C_2\to C_1$ has finite index, then $d_2$ must be injective, and therefore 
$H_2(K,\Z) = \ker d_2=0$.
\end{proof}

\subsection{Universal central extensions}
The standard reference for the following well known facts is \cite[pp. 43-47]{milnor}.  
 
 \begin{lemma}\label{l:univ} If $Q$ is a perfect group, then there is a central extension $1\to Z\to \widetilde{Q}\overset{p}\to Q\to 1$
 where
 \begin{enumerate}
 \item $\widetilde{Q}\overset{p}\to Q$ is universal: if the kernel of $E\overset{r}\onto Q$ is central in $E$, then there is a 
 unique homomorphism $f:\widetilde{Q}\to E$ such that $p=r\circ f$;
 \item $H_1(\widetilde{Q},\Z)=H_2(\widetilde{Q},\Z)=0$;
 \item if $Q$ has no non-trivial finite quotients, then neither does $\widetilde{Q}$.
 \end{enumerate}
 If $Q$ is finitely presented, then $H_2(Q,\Z)$ is finitely generated and $\widetilde{Q}$ is finitely presented. 
  \end{lemma} 
  
The following lemma from \cite{BRS} will provide a useful mechanism for forcing
quotients to have trivial second homology so that the Platonov-Tavgen criterion (Proposition \ref{p:PT})
can be applied.  This is an idea that originates in \cite{BL}.

\begin{lemma}\label{l:schur}
Let $Q$ be a perfect group with universal central extension $p:\widetilde{Q}\to Q$, let $G$ be a group with $H_2(G,\Z)=0$
and let $F:G\to Q$ be an epimorphism that restricts to $f:[G,G]\to Q$. Then, there exists an epimorphism $\tilde{f}:[G,G]\to\widetilde{Q}$
with $p\circ\tilde{f}=f$.
\end{lemma}

\begin{proof} The fibre product of the maps $F$ and $p$ is the subgroup $\breve{G}=\{(x,y)\mid F(x)=p(y)\}<G\times\tilde{Q}$.
By projecting to the first factor we see that $\breve{G}$ is a central extension
$$
0\to Z \to \breve{G} \to G\to 1,
$$
where $Z=\ker p \cong H_2(Q,\Z)$. 

Consider the standard 5-term exact sequence associated to this extension:
$$
H_2(G, \Z) \to Z \to H_1(\breve{G}, \Z) \to H_1(G, \Z) \to 0.
$$
The first term is zero by hypothesis, so $Z$ injects into $H_1(\breve{G}, \Z)$ -- in other words 
$Z \cap [\breve{G}, \breve{G}]$ is trivial. Thus $\breve{G}\onto G$ restricts to an
isomorphism  $[\breve{G}, \breve{G}]\to [G,G]$. By composing the inverse of this isomorphism
with the coordinate projection $\breve{G}\to 1\times \widetilde{Q}$ we obtain the desired 
map $\tilde{f}:[G,G]\to \tilde{Q}$, which is onto because its image contains $[\widetilde{Q}, \widetilde{Q}]$ and
$\widetilde{Q}$ is perfect.
\end{proof}

\subsection{Nested fibre products}

The following lemma allows us to bypass a discussion of the finiteness conditions
on the groups in Proposition \ref{p:PT}.
We include the proof because it provides a clear illustration of how the vanishing of $H_2(-,\Z)$ gets used.

\begin{lemma}\label{l:nested}
Let $\G$ be a group, let $p_1:\G\onto Q_1$ and $p_2:Q_1\onto Q_2$ be epimorphisms and
let $P_1<\G\times\G$ and $P_2<\G\times\G$ be the fibre products associated to $p_1$ and $p_2\circ p_1$
respectively.  If $\wh{Q}_1=1$ and $H_2(Q_2,\Z)=0$, then the inclusion $u:P_1\hr P_2$ induces an epimorphism
$\wh{u} : \wh{P}_1\to \wh{P}_2$.
\end{lemma}

\begin{proof} Let $N_1= \ker p_1\times \ker p_1$, let $N_2 = \ker p_2\circ p_1 \times 1$,  let 
$K = \ker p_2 \times 1$ and note that $K<Q_1\times Q_1$ is the image of $N_2<\G\times\G$.  

Given a finite
quotient $\pi: P_2\onto G$, we must argue that $\pi|_{P_1}$ is surjective, for which it suffices to show
that $\pi|_{N_1}$ is surjective.  If it were not,  then by taking the quotient of $G$ by the normal subgroup
$\pi(N_1)$ we would obtain a non-trivial finite quotient of the image of $P_2$ in $Q_1\times Q_1 = (\G\times \G)/N_1$.
This image $\overline{P}_2$ is the fibre product of $p_2:Q_1\onto Q_2$, which we recall is a semidirect product 
$\overline{P}_2=K\rtimes \Delta_{Q_1}$,
where $\Delta_{Q_1}<Q_1\times Q_1$ is the diagonal subgroup. 

We claim that $\overline{P}_2$ has no non-trivial
finite quotients.  To see this, first note that since $\wh{Q}_1=1$, 
any quotient would factor through $H:=(K\rtimes\Delta_{Q_1})/\<\!\< \Delta_{Q_1}\>\!\>$, which is 
isomorphic to the group of co-invariants $H_0(Q_1,\,  H_1(K,\Z))$, for the following reason:
for each $k\in \ker p_2$ we have $(k,k)\in \Delta_{Q_1}$,  so the images of $(k,1)$ and $(1,k)$
in $H$ coincide, and since $(1,k)$ commutes with $K=\ker p_2 \times 1$, the image of $K$ in 
$H$ must be abelian;  moreover we have killed the action of $\Delta_{Q_1}$ by conjugation on $K$,
and the action of $(q,q)\in \Delta_{Q_1}$ on $K_1$ is the same as the action of $(q,1)$, so
$H$ is $H_0(Q_1,\, H_1(K,\Z))$ as claimed.

To complete the proof we must argue that $H_0(Q_1,\, H_1(K,\Z))=0$,  where to save on notation
we write $Q_1$   in place of $Q_1\times 1$ and identify $K$ with $\ker p_2$.  The 5-term exact sequence
associated to the short exact sequence $1\to K\to Q_1\to Q_2\to 1$ is
$$
H_2(Q_2,\Z)\to H_0(Q_1,\,H_1(K,\Z)) \to H_1(Q_1,\Z)\to H_1(Q_2,\Z)\to 0.
$$
By hypothesis, $H_2(Q,\Z)=0$. And $H_1(Q_1,\Z)=0$ because $\wh{Q}_1=1$. Thus $H_0(Q_1,\, H_1(K,\Z))=0$.
\end{proof}

\subsection{Direct Limits}

Given a sequence of epimorphisms of  
groups $Q_1\onto Q_2\onto Q_3\onto\cdots$, we write $x_n$ for the image of each $x_1\in Q_1$ in $Q_n$.
The elements of the direct limit $Q_\infty = \dirlim  Q_i$ can then be regarded as equivalence classes of 
sequences $x_\infty = (x_1,x_2,\dots)$, where $x_\infty\sim y_\infty$ if there is an integer $n$ such that $x_n=y_n$
(which implies that $x_m=y_m$ for $m\ge n$). 

We will need the following elementary facts. We write $Z(G)$ for the centre of a group $G$.

\begin{lemma}\label{l:limits} If $Q_1$ is finitely generated and each $Q_i$ is non-trivial, then
\begin{enumerate}
\item $Q_\infty\neq 1$
\item $z_\infty\in Z(Q_\infty)$  if and only if there is an integer $n$ such that $z_n\in Z(Q_n)$
\item $x_\infty \in Q_\infty\ssm\{1\}$ has prime order $p$  if and only if there is an integer $n$ such that $x_n\in Q_n$
has order $p$
\end{enumerate}
\end{lemma}

 \section{Relatively hyperbolic groups and small-cancellation quotients}

We shall assume that the reader is familiar with Gromov's theory of hyperbolic groups \cite{gromov}. 
For the
purposes of this discussion, it is useful to recall that a finitely presented group $\G=\<A\mid R\>$ is hyperbolic
if and only if it satisfies a {\em linear isoperimetric inequality}: 
there is a constant $C$ such that, for all words $w\in F(A)$ in the free group on the generators, if $w=1$ in $\G$
then there is an equality in $F(A)$ of the form
\begin{equation}\label{e:isoper}
w=\prod_{I=1}^N \theta_i^{-1}r_i\theta_i
\end{equation}
with $\theta_i\in F,\ r_i\in R^{\pm 1}$ and $N\le C\, |w|$.

The theory of relatively hyperbolic groups was outlined by Gromov \cite{gromov} and developed by Bowditch \cite{bow},
Farb \cite{farb},   Osin \cite{osin1},  Drutu and Sapir \cite{farb}, and others.
We adopt the viewpoint of Osin \cite{osin1, osin}; this is well adapted to the
small cancellation techniques that underlie the construction of the quotients
we seek but which are hidden in our references to \cite{amo} and \cite{osin}. 
Let $\{H_\lambda\}_{\lambda\in\Lambda}$ be a collection of
proper subgroups of the group $G$ and let $A$ be a finite subset of $G$. Suppose that the natural map $\eta$ from the free product $F:=F(A)\ast(\ast_\lambda H_\lambda)$ to $G$ is surjective and that the kernel of this surjection is the normal closure of a finite set $R$. Let $\mathcal{H}= \sqcup_\lambda H_\lambda\ssm\{1\}$.
By definition, {\em $G$ is relatively hyperbolic with peripheral subgroups $\{H_\lambda\}_{\lambda\in\Lambda}$} if there is a constant $C>0$
such that every word $w$ in the alphabet $A^{\pm 1}\cup\mathcal H$ that represents the identity in $G$ satisfies an equality in $F$ with the form of (\ref{e:isoper}).  
Some authors say that $G$ is {\em properly} relatively hyperbolic to emphasize
that none of the peripheral subgroups $H_\lambda$ equals $G$, but we follow the common practice of absorbing this
into the definition.  

$G$ is said to be {\em non-elementary} if it is not virtually cyclic.

Roughly speaking, a finitely generated
group is hyperbolic relative to a system of peripheral subgroups if it acts in a controlled manner 
on a Gromov-hyperbolic metric space with conjugates of the peripheral subgroups as isotropy subgroups.
Important examples are (i) hyperbolic groups, in which case the system of peripheral subgroups is empty, and (ii) non-trivial free products $G_1\ast G_2$,
in which case the system of peripheral subgroups is $\{G_1,G_2\}$. 
A third important class of examples are
the  lattices in ${\rm{Isom}}(\H^n)$, where the peripheral subgroups are the maximal parabolic subgroups.

Theorem 1.4 of \cite{amo}, which is an application of Theorem 2.4 from \cite{osin}, 
 states that every pair of properly relatively hyperbolic groups has a common quotient that is properly relatively hyperbolic, with control on the peripheral subgroups. Theorem \ref{t:osin} is a special case of this, with some adornments that
are implicit in  \cite{amo}; cf.~\cite[Theorem 8.5]{BRS}. We are grateful to Daniel Groves for a discussion of this result and  references.

\begin{theorem}[\cite{osin}, \cite{amo}]\label{t:osin} Let $G$ be a 
non-elementary   relatively hyperbolic group and
let $H$ be a non-elementary group that is  
the free product of  finitely many (at least two) non-trivial finitely presented groups $H_i$.
Then there is an epimorphism $\mu: G\ast H\to Q$, such that
\begin{enumerate}
\item $Q$ is  a non-elementary relatively hyperbolic group;
\item if $G$ is finitely presented, then so is $Q$;
\item the restriction of $\mu$ to each of $G$ and $H$ is surjective; 
\item the restriction of $\mu$ to each $H_i$ is injective;
\item every element of finite order in $Q$ is the image of an element of finite order in $H$ or $G$;  
\item if the centre of $G$ is trivial, then the centre of $Q$ is trivial.
\end{enumerate}
\end{theorem}
  
\subsection{Rips-type complexes and $\mathfrak{d}(G)$ } 
We shall find it useful to consider the following invariant. 
 
\begin{definition}  \label{d:d}
Let $G$ be a group. Define  $\mathfrak{d}(G)$ to be
the least dimension of a contractible simplicial complex on which
$G$ can act without a global fixed point; if $G$ admits no such action, then $\mathfrak{d}(G)=\infty$.
\end{definition}

It is well-known that if $H$ is a hyperbolic group and $R>0$ is sufficiently large, then one obtains a contractible simplicial complex with
vertex set $H$ by adding a simplex spanned by each finite subset $X\subset H$ of diameter at most $R$ -- this is the Rips complex, on which $H$ acts cocompactly with finite stabilisers (see \cite[p.486]{BH}). Thus $\mathfrak{d}(H)$ is finite if $H$ is infinite and hyperbolic.
The {\em retraproducts} construction introduced in \cite{abj} produces, for each integer $d\ge 1$, a hyperbolic group $H_d$ with the property that $\mathfrak{d}(H_d)=d$.

A similar construction works in the relative case: 
working with Bowditch's definition of relative hyperbolocity,  Martinez-Pedroza and Przytycki  \cite{MP} show that every non-elementary relatively hyperbolic group   acts on a contractible simplicial complex with compact quotient and point-stabilisers that are either finite or  else conjugates of the peripheral subgroups, so $\mathfrak{d}(G)$
is finite in all the non-trivial cases.  See also \cite{dahmani}.

\section{Infinitely many profinitely-trivial quotients}

\lotsQ*

\begin{proof} Let $B$ be a finitely presented infinite group that is torsion-free and has no finite quotients; many such groups are known -- see
\cite{BG}, for example. 
We apply Theorem \ref{t:osin} with $G=\G$ and $H=B\ast B$;  let $Q_0=\overline{G}$ be the resulting common
quotient. The
theorem assures us that $Q_0$ is a non-elementary  
relatively hyperbolic group that is finitely presented if $\G$ is finitely presented.  And since $Q_0$ is a quotient of $B\ast B$, it has no non-trivial finite quotients, and neither do any of its
quotients.
Next we apply Theorem \ref{t:osin} with $G=Q_0$ and $H=C_{p_1}\ast C_{p_1}$ 
(except $H=C_2\ast C_2\ast C_2$ if $p_1=2$) to obtain $Q_0\onto Q_1$.
Then,  proceeding by induction on $i>1$, we assume that $Q_{i-1}$ has been constructed with the desired properties
and apply the construction of 
Theorem \ref{t:osin} with $G=Q_{i-1}$ and $H=C_{p_i}\ast C_{p_i}$  to obtain  $Q_{i-1}\onto Q_i$.
The theorem assures us that $Q_i$ has the desired properties.
\end{proof}

The key point in the following corollary is that $Q_i\not\cong Q_j$.

\begin{corollary}
If $\G$ is a  non-elementary  relatively hyperbolic group,  then there is an infinite sequence of
epimorphisms $G\to Q_1\onto Q_2\onto Q_3\cdots $ where each $Q_i$ is a finitely presented group with
$\wh{Q}_i=1$, and  $Q_i\not\cong Q_j$ if $i\neq j$. If $\G$ is finitely presented, then each $Q_i$
is finitely presented, and if $\G$ has trivial centre, then so does each $Q_i$.
\end{corollary}

\begin{proof} This corollary is an immediate consequence of Theorem \ref{t:lots-of-Q} except in the  
case where $\G$ has an element of order $p$ for all but finitely many primes $p$. In that case,   we have to 
build a sequence of quotients of $\G$ that are distinguished  by something other than the nature of their torsion.
The invariant that we use is the least dimension $\mathfrak{d}(Q_n)$ of a contractible simplicial complex on which $Q_n$ can act without a global fixed point.  

The retraproducts construction introduced in \cite{abj} produces, for each integer $d\ge 1$, a hyperbolic group $H_d$ with the property that $\mathfrak{d}(H_d)=d$. With these in hand, we modify the proof of Theorem \ref{t:lots-of-Q} as follows:
we form $Q_1$ as before, then we apply Theorem  \ref{t:osin} with $G=Q_1$ and $H=H_{d(1)}$
where $d(1) = \mathfrak{d}(Q_1) +1$; thus we obtain a relatively hyperbolic $Q_2$ that is a quotient of both $Q_1$ and $ H_{d(1)}$; in particular, since it is a
quotient of $H_{d(1)}$ we have $\mathfrak{d}(Q_2)\ge \mathfrak{d}(Q_1) +1$. Then, for $n\ge 2$ we proceed recursively, applying Theorem \ref{t:osin} with 
$G=Q_n$ and $H=H_{d(n)}\ast H_{d(n)}$
 where $d(n) = \mathfrak{d}(Q_n) +1$ to obtain the common quotient $Q_{n+1}$.  
\end{proof}

\subsection{Super-perfect quotients}

We want to prove Theorem \ref{t:lots-of-P}
by applying  Proposition \ref{p:PT} to the quotients $Q_i$  from Theorem \ref{t:lots-of-Q}, but in order to do so
we must arrange for $H_2(Q_i,\Z)=0$, and for this we need to impose an extra hypothesis on $\G$.

\begin{theorem}\label{t:superperfect} 
Let $\G$ be a non-elementary relatively hyperbolic group.
 Suppose that $H_2(\G,\Z)=0$.  
Let $\pi_0$ be the set of primes $p$
such that $\G$ has an element of order $p$, and let $p_1,p_2, p_3,\dots$ be a sequence of primes that
are not in $\pi_0$. Then,  there is an infinite sequence of  groups 
$\tilde{Q}_i$ and  epimorphisms $G\to \tilde{Q}_1\onto \tilde{Q}_2\onto \tilde{Q}_3\cdots $ so that,  for $i=1,2,\dots$;
\begin{enumerate}
\item  $\tilde{Q}_i$ has no non-trivial finite quotients;
\item  $H_2(\tilde Q_i,\Z)=0$;
\item if $\G$ is finitely presented, then so is each $\tilde{Q}_i$;  
\item $\tilde{Q}_i/Z(\tilde{Q}_i)$ has an element of prime order $p$  only if $p\in \pi_0\cup \{p_1,\dots, p_i\}$,  and
\item $\tilde{Q}_i/Z(\tilde{Q}_i)$ is generated by elements of order $p$ for each prime $p\in \{p_1,\dots, p_i\}$.
\end{enumerate}
\end{theorem}

\begin{proof}
We construct the epimorphisms $\G\onto Q_1\onto Q_2\onto\cdots$ 
as in Theorem \ref{t:lots-of-Q} and define $\tilde{Q}_i$ to be the universal central extension of $Q_i$. Note
that since $Q_i$ is centerless, it is equal to the quotient  $\tilde{Q}_i/Z(\tilde{Q}_i)$, in the light of Lemma \ref{l:univ},
if we can construct epimorphisms $[\G, \G]\to \tilde{Q}_1$ and $\tilde{Q}_i\onto \tilde{Q}_{i+1}$ for $i\ge 1$.

To this end,  we first observe that since $Q_1$ is perfect,   $\G\onto Q_1$ restricts to a surjection $[\G, \G]\to Q_1$.
Lemma \ref{l:schur} then tells us that we can lift this to an epimorphism $[\G, \G]\to \tilde{Q}_1$.  

Lemma \ref{l:univ} tells that  $\tilde{Q}_1=[\tilde{Q}_1, \tilde{Q}_1]$, so we can also apply Lemma \ref{l:schur} 
to the composition $\tilde{Q}_1\onto Q_1\onto Q_2$, lifting it to an epimorphism $\tilde{Q}_1\onto \tilde{Q}_2$.
Then we apply Lemma \ref{l:schur}  to the composition $\tilde{Q}_2\onto Q_2\onto Q_3$ to get 
$\tilde{Q}_2\onto \tilde{Q}_3$, and so on. 
\end{proof}

\section{Uncountable profinite ambiguity from fibre products}

The following theorem includes the awkward-sounding hypothesis that centralisers in the hyperbolic group
$H$ are  virtually cyclic.  To avoid this,
one could restrict attention to torsion-free hyperbolic groups,  but  we don't want to do this
because we want to include all cocompact lattices in ${\rm{PSL}}(2,\C)$ and all discrete non-elementary
subgroups of ${\rm{PSL}}(2,\R)$.  All of these groups have virtually cyclic centralisers for the following reason:
every non-trivial element of the group preserves a unique geodesic line in $\H^3$ or $\H^2$
or else fixes a unique point; in both
cases, the  centraliser of $h$ acts properly on the preserved set and therefore is virtually cyclic  (finite in the case
of a fixed point).

\lotsP*
 
\begin{proof} Let $\pi'$ be the set of primes that do not occur as orders of elements in $\G$. By hypothesis,
$\pi'$ is infinite, so there are uncountably many sequences of primes $p_1<p_2<p_3<\cdots$ with $p_i\in\pi'$,
and we can apply the construction of Theorem \ref{t:superperfect} to each such sequence.  
Each of the quotients
$\widetilde{Q}_i$ of $\G$ that we get from this construction
has no proper subgroups of finite index, so each of  the compositions  $G\to [\G, \G]\onto \tilde{Q}_i$ 
is surjective.  Let $P_i < G\times G$ be the fibre products associated to these surjections and note that
$P_1<P_2<P_3<\cdots$.  Proposition \ref{p:PT} tells us that each $P_i$ is finitely
generated and each of the inclusions $P_i\hr G\times G$ is a 
Grothendieck pair.  Lemma \ref{l:nested} tells us that each $P_i\hr P_{i+j}$ is a Grothendieck pair as well.  

Next we consider the direct limit $\tilde{Q}_\infty:=\dirlim \widetilde{Q}_i$.  By construction (and Lemma \ref{l:limits}), 
the quotient of this group by its centre is  ${Q}_\infty:=\dirlim {Q}_i$.   (Recall that $\widetilde{Q}_i$
was constructed as the universal central extension of the centreless group $Q_i$.)
By construction $P_\infty:=\bigcup_i P_i<G\times G$ is the fibre product of $G\onto \tilde{Q}_\infty$.
As the groups $\tilde{Q}_i$ have no non-trivial finite quotients,  $\tilde{Q}_\infty$ and ${Q}_\infty$ don't either.
And as homology commutes with direct limits,  $H_2(\tilde{Q}_\infty, \Z) =0$.  Thus Lemma \ref{l:nested} applies
and we deduce that each of the inclusions $P_i\hr P_\infty$ and $P_\infty\hr G\times G$ induce isomorphisms of
profinite completions.

To complete the proof, we must argue that the fibre products $P_i$ and $P_j$ that we have constructed are not (abstractly)
isomorphic if $i\neq j$, and we must argue that the limiting fibre products $P_\infty$ associated to different sequences of
primes are not isomorphic either.  This is where we make crucial use of the hypothesis that  $G$ 
is either a hyperbolic group  in which centralizers of non-trivial elements
are virtually cyclic,  or a central extension of such a group.  

First suppose that we are in the hyperbolic case. 
Let $N_i = (G\times 1)\cap P_i$ and $M_i = (1\times G)\cap P_i$ and note that the isomorphism type of $N_i\times M_i$
is an invariant of the abstract isomorphism type of $P_i$, because  $N_i\cup M_i$ consists
of precisely those  elements of $P_i$ whose centralizer is not virtually cyclic.  
Now,  $N_i\times M_i$ 
is the kernel of the restriction to $P_i$ of $G\times G\onto \tilde{Q}_i\times \tilde{Q}_i$,
so the isomorphism type of 
the image $P_i$ in $\tilde{Q}_i\times \tilde{Q}_i$ is also an invariant of the isomorphism type of $P_i$. 
But this image is just the diagonal copy of $\tilde{Q}_i$,  so $\tilde{Q}_i$ and 
hence $Q_i= \tilde{Q}_i/Z(\tilde{Q}_i)$ are invariants  of the abstract isomorphism  type of $P_i$.
And $Q_i$ is distinguished from $Q_j$ if $j>i$ by the fact that $Q_j$ contains torsion elements of order $p_j$
while $Q_i$ does not, by construction.

An entirely similar argument shows that the isomorphism type of $Q_\infty$ is an invariant of the abstract isomorphism
type of $P_\infty$. And by construction,  $Q_\infty$ will have $p$ torsion, for $p\in\pi'$ if and only if $p$ is one of the
primes in the sequence of primes $p_i$ used in the construction of the quotients $Q_i$.  Thus distinct sequences
of primes give non-isomorphic fibre products $P_\infty<G\times G$, and we have the uncountable collection of 
Grothendieck pairs $P\hr G\times G$ that we were seeking.

The case where $G$ has centre immediately reduces to the centreless (hyperbolic) case, 
because the centre has trivial image in each of the centreless groups $Q_i$.
\end{proof}

\section{Seifert Fibred Examples}\label{s:sfs}

We included central extensions of hyperbolic groups rather than just the hyperbolic groups themselves in 
Theorem \ref{t:lots-of-P} because this greater generality allows for interesting classes of examples. 
Prominent among these are the fundamental groups of the Seifert fibred spaces that we studied in
our work with Ryan Spitler \cite{BRS}. 
We remind the reader that a Seifert fibre space $M$ is a compact 3-manifold that is foliated by circles.  Collapsing 
the circles to points gives a map from $M$ to a 2-dimensional orbifold, where the cone points of the orbifold record
the way that certain circles in the fibration wrap around their neighbours.  Associated to this structure one has
a short exact sequence
$$
1\to \Z \to \pi_1M \to \pi_1^{\rm{orb}}(B) \to 1.
$$
We are interested in the case where the base orbifold $B$ is  
$S^2(p,q,r)$,  the quotient $\mathbb{H}^2/\Delta(p,q,r)$ of the hyperbolic plane by the 
 {\em triangle group} $\D(p,q,r) = \<a,b,c \mid a^p=b^q=c^r=1=abc\>$, 
the index-2 orientation-preserving subgroup of the reflection group associated to a hyperbolic triangle with interior angles $\pi/p$, $\pi/q$ and $\pi/r$. 

 In \cite{BMRS2} we proved that for certain values of $(p,q,r)$, the group  $\D(p,q,r)$ is profinitely
 rigid in the absolute sense.  In \cite{BRS} we proved that the fundamental groups of Seifert fibred spaces
over the corresponding orbifolds $S^2(p,q,r)$ are also profinitely rigid.  But the main result in \cite{BRS}
dealt not with these 3-manifold groups $\pi_1M$, but rather with $\pi_1M\times \pi_1M$.

Specifically, we
proved that for certain of these Seifert fibred spaces,
$\pi_1M\times\pi_1M$ is profinitely rigid in the class of all {\em finitely presented} residually finite groups,
but not in the class of all {\em finitely generated} residually finite groups.  We established the lack of 
rigidity among finitely generated groups by using 
a weaker version of Theorem \ref{t:lots-of-P} to  construct 
Grothendieck pairs $P\hr \pi_1M\times\pi_1M$.  
With the uncountable families from Theorem \ref{t:lots-of-P} in hand we can augment the main result of \cite{BRS} as follows. 
 
\begin{theorem}
\label{t:mainBRS} If $M$ is any Seifert fibred space 
with base orbifold $S^2(3,3,4)$ or $S^2(3,3,6)$ or $S^2(2,5,5)$,  then $\G=\pi_1M$ has the following properties:
\begin{enumerate}
\item $\G\times\G$ is profinitely rigid among all finitely presented, residually finite groups. 
\item For every finitely generated, residually finite group $\Pi$ with $\wh{\Pi}\cong \wh{\G\times\G}$,
there is an embedding $\Pi\hookrightarrow\G\times\G$ that induces the isomorphism $\wh{\Pi}\cong \wh{\G\times\G}$.
\item There exist uncountably many non-isomorphic  groups $\Lambda$ with embeddings
 $\Lambda\hookrightarrow\G\times\G$ that induce an  isomorphism $\wh{\Lambda}\cong \wh{\G\times\G}$;
\item infinitely many of these groups $\Lambda$ are finitely generated. 
\end{enumerate}  
\end{theorem}

\section{Kleinian Examples: Galois rigidity and Theorem  \ref{t:fig8}}\label{s:galois}

In this section we will sketch a proof Theorem \ref{t:fig8}.
We shall see that Theorem \ref{t:lots-of-P} can be applied to prove items (3) and (4) of Theorem \ref{t:fig8},
because the groups $\G$ in Theorem \ref{t:fig8} have finite index in hyperbolic orbifold groups with trivial second homology,
and they are the commutator subgroups of these orbifold groups. First, though, we explain 
why  for   Kleinian groups such as $\G$,   the range of groups in the  profinite genus of $\G\times\G$ is entirely
accounted for by Grothendieck pairs $P\hr \G\times\G$.   This phenomenon is related to the theory of
{\em Galois rigidity} that was developed in \cite{BMRS1} and was  adapted to direct products in \cite{BRS}.
We shall not rehearse this theory in detail, but we will recall the main points and indicate why
they apply to the groups in Theorem \ref{t:fig8}.  

\subsection{Trace fields and Galois rigidity} \label{s:Galois}
 
Let $\phi\colon \mathrm{SL}(2,\mathbb{C}) \to \mathrm{PSL}(2,\mathbb{C})$ be the quotient homomorphism, and if $H$ is a finitely generated subgroup of $\PSL(2,\mathbb{C})$,  set $H_1 = \phi^{-1}(H)$. It will be convenient to say that $H$ is 
{\em Zariski-dense} in $\PSL(2,\C)$ when what we actually mean is that $H_1$ is a Zariski-dense subgroup of $\SL(2,\C)$. The \textit{trace-field} of $H$ is defined to be the field 
\[ K_H=\mathbb{Q}(\mathrm{tr}(\gamma)~\colon~ \gamma \in H_1). \] 
If $K_H$ is a number field with ring of integers $R_{K_H}$, we say that $H$ has {\em integral traces} if $\tr(\gamma)\in R_{K_H}$ for all $\gamma \in H_1$.  

An important object associated to $H$ is the subalgebra of ${\rm{Mat}}(2,\C)$ generated by $H_1$; this is 
a {\em quaternion algebra} over $K_H$.

Suppose now that $\La$ is an abstract finitely generated group and  $\rho\colon \La\rightarrow \PSL(2,\C)$ a Zariski-dense representation with $K=K_{\rho(\La)}$ a number field of degree $n_K$. If $K=\Q(\theta)$ for some algebraic number $\theta$, then the Galois conjugates of $\theta$, say $\theta=\theta_1,\dots,\theta_{n_K}$, provide embeddings $\sigma_i\colon K\to\C$ defined by $\theta\mapsto\theta_i$.  These in turn can be used to build $n_K$ Zariski-dense non-conjugate representations $\rho_{\sigma_i}\colon \La \to \PSL(2,\C)$ with the property that $\tr(\rho_{\sigma_i}(\gamma))=\sigma_i(\tr\rho(\gamma))$ for all $\gamma\in \La$. 
One refers to these as {\em Galois conjugate representations}. 
The existence of these Galois conjugates shows that $|\mathrm{X}_{\mathrm{zar}}(\G,\mathbb{C})|\geq  n_{K_{\rho(\La)}}$, where  $\mathrm{X}_{\mathrm{zar}}(\La,\C)$ denotes the set of Zariski-dense representations $\G\to\PSL(2,\C)$ up to conjugacy.
  
\begin{definition}\label{def:galois-rigid}
Let $\La$ be a finitely generated group and $\rho\colon \La\to \PSL(2,\C)$ a Zariski-dense representation whose trace field $K_{\rho(\La)}$ is a number field. If
 $|\mathrm{X}_{\mathrm{zar}}(\La,\mathbb{C})|= n_{K_{\rho(\G)}}$, we say that $\G$ is {\em Galois rigid}  (with associated field $K_{\rho(\La)}$).
\end{definition}

\subsection{From Galois rigidity to profinite rigidity}\label{s:gr2pr}
In outline, the way in which
the groups $\G<{\rm{PSL}}(2,\C)$ in \cite{BMRS1} were proved to be profinitely rigid was as follows.  To begin, 
$\G$ needs to have only fiinitely many Zariski-dense representations into ${\rm{PSL}}(2,\C)$ and one needs
to understand these well enough to see that $\G$  is Galois rigid.  It is also important that 
all traces in these representations are
integers in the number field $K_\G$. 
The integrality of traces is used to ensure the boundedness of the non-archimedean representations  
$\G\to {\rm{SL}}(2,\-{\Q_p})$ obtained by  transporting Zariski dense representations 
$\G\to {\rm{SL}}(2,\C)$  via (non-continuous) field isomorphisms $\C\cong\-{\Q_p}$, for all primes $p$. This boundedness enables one to extend the representations to $\wh{\G}$,  then restrict to any finitely generated $\Lambda$ with
$\wh{\Lambda}\cong\wh{\G}$ to obtain bijections between the sets of bounded, Zariski-dense  characters of $\G$ 
and $\Lambda$ at every finite place (Lemma 4.6 of \cite{BMRS1}). 

With this control established, one can argue
that the abstract group $\Lambda$ is Galois rigid with a Zariski-dense representation $\rho\colon \La\to \PSL(2,\C)$
whose arithmetic data match those  of $\G$ -- see \cite[Theorem 4.8]{BMRS1} for a precise statement. In good situations
(which include the low-degree number fields of our examples),  this matching of data forces the number fields
$K_{\rho(\Lambda)}$ and $K_\G$ to be equal,  and likewise the quaternion algebras of $\G$ and $\rho(\Lambda)$.
 This in turn allows one to
conclude that (up to Galois conjugation) the image of $\rho:\La\to  \PSL(2,\C)$ is contained in $\G$ or a small
extension of it (we'll be a bit more specific about this finite extension in the next section).
In the examples where the argument concludes successfully,
 one can force the image to lie in $\G$ using specific information about the lattices in
${\rm{PSL}}(2,\C)$ that lie between $\G$ and the finite extension in question.

One then has to argue that every non-elementary subgroup $S<\G$ has a finite quotient that $\G$ does not
have \cite{prasad}, which means that $\wh{\Lambda}\cong\wh{\G}$ cannot map onto $\wh{S}$, 
from which we conclude that
$\rho(\Lambda)=\G$.   
The Hopf property for finitely generated profinite groups then tells us that $\wh{\rho}$ is
injective, and hence $\rho:\Lambda\to\G$ is an isomorphism.

\subsection{Extension to direct products}\label{s:dirProd}
 Because centralisers of non-trivial elements in Zariski-dense subgroups of
${\rm{PSL}}(2,\C)$ are virtually abelian,  if $\G$ is a lattice in ${\rm{PSL}}(2,\C)$ then the Zariski-dense representations
of $\G\times\G$ break into two families: those that factor through the projection to the first factor, and those that
factor through the projection to the second factor.  Correspondingly,   
 if $\G$ is Galois rigid with integral
traces, as in section (\ref{s:gr2pr}),  and if  $\Lambda$ is a finitely  generated, residually finite group with 
$\wh{\Lambda}\cong \wh{\G\times\G}$,  then one can prove that the Zariski-dense representations of $\Lambda$
into ${\rm{PSL}}(2,\C)$ break into two families,  each consisting to the representations that factor through one of
the coordinate projections of $\Lambda$ in $\wh{\Lambda}\cong \wh{\G\times\G}$.   This is explained in detail in \cite{BRS}.

If $\G$ is one of the groups for which one can deduce profinite rigidity from Galois rigidity because 
\cite[Theorem 4.8]{BMRS1} applies, then by applying the same argument to the projection of $\Lambda$ 
to each coordinate in  $\wh{\Lambda}\cong \wh{\G\times\G}$ (being careful with some matching of data)
one  can deduce that the image $\Lambda_1$
 of $\Lambda$ in $\wh{\G}\times 1$,  and the image $\Lambda_2$ in $1\times\wh{\G}$, are both isomorphic
 to $\G$ (see  \cite[Section 3]{BRS}). 
  By composing with an isomorphism of $\wh{\G\times\G}$, we may assume that $\Lambda_1$
 and $\Lambda_2$ are equal to $\G\times 1$ and $1\times\G$.  
  It follows that the original isomorphism $\wh{\Lambda} \cong \wh{\G\times\G}$ is induced
  by an inclusion $\Lambda\hr \G\times\G$. We refer the reader to \cite{BRS} for more details.
 
This concludes our explanation of the theory behind
Theorem \ref{t:fig8}(2). In the following two sections we will say a little more about
why this outline applies to the Weeks group $\GW$ and the fundamental group $\G_4$ of the 4-fold cyclic  
branched cover of the figure eight knot.  
 
 \subsection{The Weeks manifold}\label{s:weeks}

Let  $M_W={\mathbb H}^3/\Gamma_W$ denote the Weeks manifold; this is the unique closed orientable hyperbolic 3--manifold of minimal volume.  It is arithmetic and  
can be obtained by performing  $(-5,1)$ surgery on one component of the Whitehead link and $(5,2)$ surgery on the other.  The arithmetic structure of the Weeks manifold $M_W$ is described in \cite[Ch 4.8.3, Ch 9.8.2]{MR}): the invariant trace-field of 
$\G_W$ is ${\mathbb Q}(\theta)$ where $\theta^3-\theta^2+1=0$; this is a field of discriminant $-23$. Amongst other things, arithmeticity implies that $\Gamma_W$ has integral traces.   

Using SnapPy \cite{CDW}, a presentation for $\GW$ can be computed:
$$\<a,b\mid ababa^{-1}b^2a^{-1}b,abab^{-1}a^2b^{-1}ab\>.$$
From this, we see that $H_1(M_W,{\mathbb Z})\cong {\mathbb Z}/5{\mathbb Z}\times {\mathbb Z}/5{\mathbb Z}$.

It is shown in  \cite[Cor 6.3]{ReWan} that $\G_W$ has only three characters of irreducible representations,
into ${\rm{PSL}}(2,\C)$,  namely the characters of the faithful discrete representation, its complex conjugate, and a $\PSU(2)$--representation arising from the real ramified place of $B_W$.
It follows that $\GW$ is Galois rigid, and from above, has integral traces.  
It is explained in \cite[Corollary 4.11 and Example  4.13]{BMRS1} 
why the structure of  this field and
the invariant quaternion algebra $B_W$ allow one to follow the template for establishing profinite
rigidity that is sketched in Section \ref{s:gr2pr}, and hence why Section \ref{s:dirProd} applies.

The ``small finite extension" of $\GW$ that was alluded to in (\ref{s:gr2pr}) is the group of units $\G_\mathcal{O}^1$
for a choice of maximal order $\mathcal{O}$ in the quaternion algebra $B_W$.  The Weeks group $\GW$ is a normal
subgroup of index $3$ in $\G_\mathcal{O}^1$.  
According to  \cite{MedV}, the orbifold $\H^3/\G_\mathcal{O}^1$ can be described as $(3,0)$ Dehn surgery on $5_2$ (i.e.~$M_W$ is the $3$-fold cyclic branched cover of the knot $5_2$). 

The following lemma tells us that Theorem \ref{t:lots-of-P} applies to $\GW$.

\begin{lemma}\label{l:weeks-commut}
$\GW$ is the commutator subgroup of $\G_\mathcal{O}^1$, and $H_2(\G_\mathcal{O}^1,\Z)=0$.
\end{lemma}

\begin{proof} As with any knot,  the  first homology of $\S^3\smallsetminus 5_2$ is
generated by a meridian curve $\mu\in \pi_1(\S^3\smallsetminus 5_2)$.  
As $5_2$ is a 2-bridge knot,  it has a presentation with 2 generators and 1 relation.   The orbifold group
$\G_\mathcal{O}^1$ is obtained by adding the relation $\mu^3=1$ to  
$\pi_1(\S^3\smallsetminus 5_2)$, so $\G_\mathcal{O}^1$ has a presentation 
with 2 generators and 2 relations and its abelianization is cyclic of order $3$.  
It follows  from Lemma \ref{l:balanced} 
that $H_2(\G_\mathcal{O}^1,\Z)=0$. 

As $\GW< \G_\mathcal{O}^1$
is normal of index 3, it must coincide with the commutator subgroup.
\end{proof}

\begin{remark}
In \cite{BMRS1} we also established that the group $\G_\mathcal{O}^1$ and its normalizer in $\PSL(2,\C)$ (denoted by $\Gamma_{\mathcal{O}}$) are profinitely rigid. Moreover, $\G_\mathcal{O}^1 = [\Gamma_{\mathcal{O}},\Gamma_{\mathcal{O}}]$,
and so in principle we can apply the same argument as we did above for $\GW$.  Unfortunately, we do not know that $H_2(\Gamma_{\mathcal{O}},\Z)=0$. In addition, \cite{MedV} provides the lattice of super groups between $\GW$ and $\Gamma_{\mathcal{O}}$, and although we
can prove profinite rigidity of certain of these using \cite{prasad} (e.g the orbifold groups of $9_{49}(2)$ and $7^2_1(2,3)$ in the notation of \cite{MedV}) we are unable to fulfill the hypothesis of Theorem \ref{t:lots-of-P} to build other uncountable familes.
\end{remark}
 
\begin{remark} In our previous work \cite{BRS},   we showed that if $\Gamma$ is one of the Seifert groups of Theorem \ref{t:mainBRS} then there are {\em{no}} Grothendieck pairs $P\hookrightarrow \Gamma\times\Gamma$ 
with {\em $P$ finitely presented}.  In contrast,  we do not yet know if this holds in the case of $\GW\times\GW$.
\end{remark}

\subsection{A branched cover of the figure-eight knot} 
\label{fig8_section}

Let $K\subset S^3$ denote the figure-eight knot, let  $Q_n$ be the orbifold obtained by 
$(n,0)$-Dehn filling on $K$ and let $M_n$ be the $n$-fold cyclic branched cover of $K$.
Note that $M_n$ can also be regarded as an $n$-fold cyclic (orbifold) cover of $Q_n$; this is
the maximal abelian cover.
 When $n\geq 4$,  the manifold $M_n$ and the orbifold $Q_n$ are hyperbolic 
(see \cite{CHK} for example).  Let $\Gamma_n=\pi_1M_n$.  An explicit presentation of $\G_4$ is
$$
\G_4=\<a,b \mid b a^{-2} b a^{-1} b^2 a b^2 a^{-1},    a^2 b a b^2 a b a^2 b^{-1}\>.
$$

It is proved in \cite[Section 4.1]{Re0} that $\G_4$ is an arithmetic Kleinian group with (invariant) trace-field $k=\Q(\sqrt{-3})$ 
and that its invariant quaternion algebra $B$ over $k$  is ramified at the places $\nu_2$ and $\nu_3$ associated to the prime ideals of norm $4$ and $3$, respectively, in $k$.    We  explained in \cite{prasad} why this arithmetic information,
together with detailed information about the small number of lattices in ${\rm{PSL}}(2,\C)$ that contain $\G_4$,
is sufficient to implement the basic strategy from  \cite{BMRS1} for establishing profinite rigidity.
This  allows us to follow the analysis of profinite equivalences $\wh{\Lambda}\cong \wh{\G}_4\times \wh{\G}_4$
outlined in section \ref{s:gr2pr} and deduce, as we did for the Weeks manifold, that Theorem \ref{t:fig8}(2)
holds for $\G_4$. To complete the proof of Theorem \ref{t:fig8}
we need a substitute for Lemma \ref{l:weeks-commut}.  The proof of the following lemma is entirely
similar to that of Lemma \ref{l:weeks-commut}: once again we have an orbifold group obtained by killing a
power of the generator of homology in a 2-bridge knot group. In this
case, $\G_4 = [\Delta_4, \Delta_4]$ has index $4$ in $\Delta_4$.

\begin{lemma}\label{l:G4-commut}
$\G_4$ is the commutator subgroup of the fundamental group $\Delta_4$ of the
orbifold $Q_4$ described above,  and  $H_2(\Delta_4,\Z)=0$.
\end{lemma}
 
We also showed in \cite{prasad} that the manifold known as $\rm{Vol}(3)$ has profinitely rigid fundamental group, as well as the corresponding group $\G_\mathcal{O}^1$ and maximal group $\Gamma_\mathcal{O}$ (for a maximal order 
$\mathcal{O}$ in the quaternion algebra $B$ as above).  
However, as with the supergroups of $\GW$ described above, we are unable to arrange for the hypotheses of Theorem \ref{t:lots-of-P} to hold for these groups. 

\subsection{Other examples}

We close by pointing out some families of Kleinian groups, which, whilst not replicating all of Theorem \ref{t:fig8}, do  allow us to prove some additional partial results which we summarized in the following theorem.

\begin{theorem}
\label{more_Kleinian}
Let $K\subset S^3$ be a hyperbolic knot,  let $X_n$ be the $n$-fold cyclic branched cover of $K$, and let $\Gamma_n = \pi_1(X_n)$.  If $n>3$, then, $X_n$ is hyperbolic and 
there exist uncountably many non-isomorphic  groups $P_\lambda$ with embeddings $P_\lambda\hookrightarrow \G_n\times \G_n$
that induce an isomorphism of profinite completions.  

Moreover, when $K$ is the figure-eight knot and $n$ is prime,  for every finitely generated, residually finite group $\Lambda$, for with  $\wh{\Lambda}\cong\wh{\G}_n\times\wh{\G}_n$  
there exists an embedding $\Lambda\hookrightarrow \G_n\times\G_n$ 
inducing this isomorphism of profinite completions.\end{theorem}

\begin{proof} That $X_n$ is hyperbolic is a consequence of the Orbifold Theorem (see \cite{CHK}). As in the discussion in section \ref{fig8_section}, $\Gamma_n$ is the commutator subgroup of the orbifold group $H_n$ obtained by 
$(n,0)$ Dehn surgery on $K$.
The  group of any non-trivial knot in $S^3$ has deficiency one (as one can see from a Wirtinger presentation),  so adjoining the relation $\mu^n$ where $\mu$ is a meridian of $K$ we get a balanced 
presentation of the orbifold group $H_n$.   
From Lemma \ref{l:balanced} we deduce that  $H_2(H_n,\Z)=0$.  Thus we can apply Theorem \ref{t:lots-of-P} 
to obtain the  statement in the first paragraph of the theorem.

For the assertion in the second paragraph,
we note that it was shown in \cite{prasad} that when $K$ is the figure-eight knot and $n>3$ prime, the groups $\G_n$ are Galois rigid. Hence, the discussion in section \ref{s:dirProd} can be applied to the groups $\G_n$ to complete the proof.
\end{proof}

\section{Some questions}
\label{s:finite}

We conclude by collecting some questions that arise out of our work.

\begin{question}
Which hyperbolic triangle groups $\D(p,q,r)$ admit a surjection to a finitely presented infinite group $Q$ with
$\wh{Q}=1$ and $H_2(Q,\Z)=0$?
\end{question}

\begin{question} For the groups $\GW$ abd $\G_4$ from Theorem \ref{t:fig8}, 
do there exist Grothendieck pairs $P\hr \GW\times\GW$ (or $P\hr \G_4 \times\G_4$) with $P$ finitely presented?
\end{question}

\begin{question}
Do there exist non-elementary hyperbolic groups $H$ for which the strong profinite genus of $H\times H$ is trivial,
i.e.  the group $H$ has the property that there does not exist any Grothendieck pairs $P\hr H\times H$ with $P\neq H\times H$ finitely generated  
(a property that is called {\em left} Grothendieck rigid in \cite{JL}).
\end{question} 

As the reader will have noted, many  of the issues in the current paper and \cite{BRS} involve arranging for the
condition $H_2(\G,\Z)=0$ to hold.  This prompts that following
question, a positive answer to which would simplify various constructions herein and in \cite{BRS}.

\begin{question} Does there exist an integral 
homology 3-sphere $\Sigma$ 
with infinite fundamental group such that $\pi_1\Sigma$ is profinitely rigid?
\end{question}


\end{document}